\documentclass{amsart}
     
\newtheorem{theorem}{Theorem}[section]
\newtheorem{lemma}[theorem]{Lemma}

\theoremstyle{definition}

\theoremstyle{remark}
\numberwithin{equation}{section}

\newcommand{\abs}[1]{\lvert#1\rvert}

\def\a{{\alpha}}
\def\b{{\beta}}
\def\g{{\gamma}}

\def\e{{\epsilon}}
\def\ep{{\eta}}

\def\l{{\lambda}}
\def\s{{\sigma}}

\def\na{{\nabla}}

\def\l{{\lambda}}

\def\G{{\Gamma}}
\def\O{{\Omega}}
\def\D{{\Delta}}

\def\xxx{{{\bf x}}}

\def\R{{{\bf R}^1}}

\def\RN{{{\bf R}^N}}

\def\9{{\ \hbox{in}\ \O}}
\def\1{{\ \hbox{on}\ \G_1}}
\def\2{{\ \hbox{on}\ \G_2}}
\def\3{{\ \hbox{on}\ \G_3}}
\def\0{{\ \hbox{on}\ \G}}

\begin{document}

\title[ The exponential decay for the damped wave equation]{Remark on the exponential decay of the solutions of the damped wave equation}
\author{Giovanni Cimatti}
\address{Department of Mathematics, Largo Bruno
  Pontecorvo 5, 56127 Pisa Italy}
\email{cimatti@dm.unipi.it}


\subjclass[2010]{ 35B35}



\keywords{Damped wave equation, Gronwall's lemma, Exponential decay, Coefficient of exponential decay }

\begin{abstract}
A condition which guaranties the exponential decay of the solutions of the initial-boundary value problem for the damped wave equation is proved. A method for the effective computability of the coefficient of exponential decay is also presented.
\end{abstract}

\maketitle

\section{introduction}
Let $\O$ be a bounded and open subset of $\RN$ and $\G$ its boundary. Consider the initial-boundary value problem

\begin{equation}
\label{1_1}
u_{tt}-\D u+au_t=0\quad\hbox{in}\quad \O\times(0,\infty)
\end{equation}

\begin{equation}
\label{2_1}
u=0\quad\hbox{on}\quad \G\times[0,\infty)
\end{equation}

\begin{equation}
\label{3_1}
u(\xxx,0)=u_o(\xxx)\quad\hbox{for}\quad\xxx\in\O
\end{equation}

\begin{equation}
\label{4_1}
u_t(\xxx,0)=u_1(\xxx)\quad\hbox{for}\quad\xxx\in\O,
\end{equation}
where $a$ is a positive constant, $u_o(\xxx)$ and $u_1(\xxx)$ are given functions satisfying the usual regularity and compatibility conditions implied in the assumption we make that $u(\xxx,t)$ is a classical solution of problem (\ref{1_1})-(\ref{4_1}). P.H. Rabinowitz has proven \cite{R} that the "energy" i.e. $\int_\O\abs{\na u(\xxx,t)}^2dx$ associated with the solution of (\ref{1_1})-(\ref{4_1}) decays exponentially as $t\to\infty$. A different proof of this fact, important in applications, has been given by R. Temam \cite{T}, \cite{T1} and by R Temam and J.M. Ghidaglia in \cite{TG}. In this paper we consider the initial-boundary value problem  (\ref{1_1})-(\ref{4_1}) with the more general equation

\begin{equation}
\label{1_2}
u_{tt}-\D u+\s(\xxx,t)u_t=0\quad\hbox{in}\quad \O\times(0,\infty),
\end{equation}

where $\s(\xxx,t)$ is supposed to be continuous and positive in $\O\times (0,\infty)$. In general there is no exponential decay, as the following example shows. Let $N=1$, $\O$=$(1,2)$ and $\s(\xxx,t)=\frac{2t}{x^2-3x+2}$. The solution is now given by $u(\xxx,t)=t(x^2-3x+2)$ for which there is no exponential decay. We prove, with fully elementary means, and without using the semigroup's theory, that a sufficient condition for having the exponential decay is

 \begin{equation}
\label{1_4}
0<\s_0\leq\s(\xxx,t)\leq\s_1<\infty,
\end{equation}

where $\s_0=\inf\s(\xxx,t)$ and  $\s_1=\sup\s(\xxx,t)$. We also give an effective way to compute the decay exponent in terms of the parameters $\s_0,\s_1,\l_1$, where $\l_1$ is the first eigenvalue of the laplacian in $\O$ with zero boundary conditions.

\section{exponential decay}
Let $u(\xxx,t)$ be a regular solution of (\ref{1_2}), (\ref{2_1}), (\ref{3_1}) and (\ref{4_1}) with $\s(\xxx,t)$ satisfying (\ref{1_4}). Define

 \begin{equation}
\label{1_5}
v(\xxx,t)=u_t(\xxx,t)+\e u(\xxx,t)\quad\hbox{with}\quad 0<\e\leq\s_0.
\end{equation}

We have

 \begin{equation}
\label{3_5}
u_{tt}=v_t-\e v+\e^2u.
\end{equation}

Substituting (\ref{3_5}) in (\ref{1_2}) we obtain 

 \begin{equation}
\label{4_5}
v_t+(\s-\e)v+(\e^2-\e\s)u-\D u=0.
\end{equation}

Let us multiply (\ref{4_5}) by $v$. We arrive at

 \begin{equation}
\label{1_6}
\int_\O v_tvdx+\int_\O(\s-\e)v^2dx-\int_\O v\D udx+\int_\O(\e^2-\e\s)uvdx=0.
\end{equation}

Since $v=0$ on $\G\times [0,\infty)$, integrating by parts in (\ref{1_6}) we have

 \begin{equation}
\label{3_6}
\frac{1}{2}\frac{d}{dt}\int_\O v^2dx+\int_\O (\s-\e)v^2dx+\int_\O\na v\cdot\na u dx-\e\int_\O(\s-\e)uvdx=0.
\end{equation}

By (\ref{1_5}) $\na v=\na u_t+\e\na u$, hence from (\ref{3_6}),

 \begin{equation}
\label{1_7}
\frac{1}{2}\frac{d}{dt}\int_\O \big(\abs{\na u}^2+v^2\bigl)dx+\int_\O (\s-\e)v^2dx+\e\int_\O\abs{\na u}^2dx-\e\int_\O(\s-\e)uvdx=0.
\end{equation}

Recalling (\ref{1_5}) we have

\begin{equation}
\label{3_7}
\int_\O (\s-\e)v^2dx\geq\int_\O(\s_0-\e)v^2dx\geq 0.
\end{equation}

To control from below the last integral in (\ref{1_7}) we estimate from above $\e\int_\O(\s-\e)uvdx$. Let $\l_1>0$ be the first eigenvalue of the laplacian with zero boundary condition on $\G$. Using the Cauchy-Schwartz and Poincarè inequalities we obtain, with $\ep>0$,

\begin{equation}
\label{1_9}
\e\int_\O(\s-\e)uvdx\leq\frac{\e(\s_1-\e)}{2\ep}\int_\O\abs{v}^2dx+\frac{\e(\s_1-\e)\ep}{2\l_1}\int_\O\abs{\na u}^2dx.
\end{equation}

Changing sign in (\ref{1_9}) and substituting in (\ref{1_7}) we obtain

\begin{equation}
\label{2_9}
\frac{1}{2}\frac{d}{dt}\int_\O\bigl(\abs{\na u}^2+v^2\bigl)dx+\biggl[\e+\frac{\e(\e-\s_1)\ep}{2\l_1}\biggl]\int_\O\abs{\na u}^2dx+\biggl[\s_0-\e+\frac{\e(\e-\s_1)}{2\ep}\biggl]\int_\O\abs{v}^2dx\leq 0.
\end{equation}

Defining

\begin{equation}
\label{1_10}
f(\e,\ep,\s_1,\l_1)=\e+\frac{\e(\e-\s_1)\ep}{2\l_1},\quad g(\e,\ep,\s_0,\s_1)=\s_0-\e+\frac{\e(\e-\s_1)}{2\ep}
\end{equation}

(\ref{2_9}) becomes

\begin{equation}
\label{3_10}
\frac{1}{2}\frac{d}{dt}\int_\O\bigl(\abs{\na u}^2+v^2\bigl)dx+f(\e,\ep,\s_1,\l_1)\int_\O\abs{\na u}^2dx+g(\e,\ep,\s_0,\s_1)\int_\O\abs{v}^2dx\leq 0.
\end{equation}

We wish to find couples $(\e,\ep)$ satisfying the conditions $0<\e<\s_0$ and $0<\ep$ such that for every choice of the parameters $\s_0,\s_1,\l_1$, $f(\e,\ep,\s_,\l_1)$ and $g(\e,\ep,s_0,\s_1)$ are positive and equal. Among these couples we must find the $(\e^*,\ep^*)$ which gives to $f$ (and $g$) the greatest possible value. This will permits to apply to (\ref{3_10}) the Gronwall's inequality. The exponential decay with the best decay exponent then follows easily. Let us consider in the plane $\e$, $\ep$, ($\e>0$) the families of curves

\begin{equation}
\label{1_12}
f(\e,\ep,\s_1,\l_1)-g(\e,\ep,\s_0,\s_1)=0.
\end{equation}

Since $\e\l_1>0$, (\ref{1_12}) is equivalent to

\begin{equation}
\label{2_12}
\e(\e-\s_1)\ep^2+2\l_1(2\e-\s_0)\ep+\l_1\e(\s_1-\e)=0.
\end{equation}

Solving (\ref{2_12}) with respect to $\ep$ we obtain two branches of solutions. The one of interest to us is

\begin{equation}
\label{1_13}
\ep=\frac{\sqrt{(\s_0-2\e)^2\l_1^2+\e^2(\s_1-\e)^2\l_1}+(2\e-\s_0)\l_1}{\e(\s_1-\e)}.
\end{equation}

Inserting (\ref{1_13}) in $f(\e,\ep,\s_1,\l_1)$ (or in $g(\e,\ep,\s_0,\s_1)$) we obtain

\begin{equation}
\label{2_13}
F(\e,\s_0,\s_1,\l_1)=\frac{\s_0}{2}-\frac{\sqrt{(\s_0-2\e)^2\l_1^2+\e^2(\s_1-\e)^2\l_1}}{2\l_1}.
\end{equation}

We study

\begin{equation}
\label{2_13}
\a=F(\e,\s_0,\s_1,\l_1)
\end{equation}
as a function of $\e$ depending of the three parameters $\s_0,\s_1,\l_1$. We have

\begin{lemma}
Let $\s_1>\s_0>0$ and $\l_1>0$. If $\bar\e$ is a solution of the equation $F(\e,\s_0,\s_1,\l_1)=0$, then $\bar\e<\s_0$.
\end{lemma}

\begin{proof}
By contradiction, suppose $\bar\e=\s_0+\g^2$, $\g\neq 0$ is a solution. We have

\begin{equation}
\label{2_14}
\s_0^2\l_1^2=(-\s_0-2\g^2)^2\l_1^2+(\s_0+\g^2)^2(\s_1-\s_0-\g^2)^2\l_1>(\s_0+2\g^2)^2\l_1^2>\s_0^2\l_1^2.
\end{equation}
On the other hand, if $\bar\e=\s_0$, we have

\begin{equation}
\s_0^2\l_1^2=\s_0\l_1^2+\s_0^2(\s_1-\s_0)^2\l_1>\s_0^2\l_1^2.
\end{equation}
\end{proof}

For every values of the parameters $\e=0$ is a solution of $F(\e,\s_0,\s_1,\l_1)=0$. Moreover, $\lim_{\e\to \pm\infty}F(\e,\s_0,\s_1,\l_1)=-\infty$ and

\begin{equation}
\label{1_15}
F'(\e,\s_0,\s_1,\l_1)=\frac{2(\s_0-2\e)\l_1-\e(\s_1-\e)^2+\e^2(\s_1-\e)}{2\sqrt{(\s_0-2\e)^2\l_1^2+\e^2(\s_1-\e)^2\l_1}}.
\end{equation}

We have $F'(0,\s_0,\s_1,\l_1)=1$ for every value of the parameters $\s_0$, $\s_1$ and $\l_1$. Therefore, in a small interval $(0,\b)$, $\b>0$ $F$ is positive for every value of the parameters. To study completely $F'$ we make the following elementary discussion. The cubic expression

\begin{equation}
\label{1_16}
-2\e^3+3\s_1\e^2-(4\l_1+\s_1^2)\e+2\l_1\s_0
\end{equation}

has the same sign of $F'(\e,\s_0,\s_1,\l_1)$. On the other hand, $D=3\s_1^2-24\l_1$ is the discriminant of the derivative of (\ref{1_16})  Thus, if $D<0$, (\ref{1_15}) is always strictly decreasing. As a consequence, $F'(\e,\s_0,\s_1,\l_1)$ vanishes in exactly one point $\bar\e$ and $0<\bar\e$. This implies that $F(\e,\s_0,\s_1,\l_1)$, which always vanishes for $\e=0$ and is positive immediately to the right of $\e=0$, has a positive absolute maximum in $\e^*$ and vanishes in $\e=\bar\e>\bar\e^*$. If $D=0$ a bifurcation occurs, and, when $D>0$, $F(\e,\s_0,\s_1,\l_1)$ has two relative maxima's and one relative minimum. The maximum immediately to the right of $\e=0$ is certainly positive, the second maximum may or may not be positive. Let $\e^*$ be the point of absolute maximum of $F(\e,\s_0,\s_1)$. Define

\begin{equation}
\label{1_19}
\a^*=F(\e^*,\s_0,\s_1,\l_1)
\end{equation}

and

\begin{equation}
\label{2_19}
\ep^*=\frac{\sqrt{(\s_0-2\e^*)^2\l_1^2+\e^{*2}(\s_1-\e^*)^2\l_1}+(2\e^*-\s_0)\l_1}{\e^*(s_1-\e^*)}.
\end{equation}

With this choice of $\e$ and $\ep$ we have $f(\e^*,\ep^*,\s_1,\l_1)=g(\e^*,\ep^*,\s_0,\s_1)=\a^*>0$. Therefore, (\ref{3_10}) becomes

\begin{equation}
\label{2_20}
\frac{1}{2}\frac{d}{dt}\int_\O\bigl(\abs{\na u}^2+v^2\bigl)dx+\a^*\int_\O(\abs{\na u}^2+v^2)dx\leq 0.
\end{equation}

Using the Gronwall's inequality, we obtain

\begin{equation}
\label{3_20}
\int_\O\bigl(\abs{\na u}^2+v^2\bigl)dx\leq\int_\O(\abs{\na u(\xxx,0)}^2+\abs{v(\xxx,t)}^2)dx\ e^{-2\a^* t}.
\end{equation}

The right hand side of (\ref{3_20}) can be computed in terms of the initial and boundary data satisfied by $u(\xxx,t)$. For, from (\ref{3_1}) we obtain $\na u(\xxx,0)=\na u_0(\xxx)$. Moreover, since $v(\xxx,t)=u_t(\xxx,t)+\e^*u(\xxx,t)$, we have $v(\xxx,0)=u_1(\xxx)+\e^*u_0(\xxx)$. Hence

\begin{equation*}
\int_\O\abs{\na u}^2dx\leq\int_\O\abs{\na u_0(\xxx)}^2+\abs{u_1(\xxx)+\e^*u_0(\xxx)}^2dx\ e^{-2\a^*t}.
\end{equation*}

This proves the exponential decay.
\bigskip

\noindent{\bf Remark}. For the regular solutions of the non-linear equation

\begin{equation}
\label{1_21}
u_{tt}-\D u+m(u)u_t=0\quad\hbox{in}\quad \O\times(0,\infty)
\end{equation}
the exponential decay can be obtained, with minor changes, if we keep the initial and boundary conditions (\ref{2_1}), (\ref{3_1}) and (\ref{4_1}) and assume $m\in C^0(\R)$ and $0<m_0\leq m(u)\leq m_1$. To see that, we simply define $\s(\xxx,t)=m(u(\xxx,t))$ where $u(\xxx,t)$ is the regular solution of the non linear problem 
(\ref{1_21}), (\ref{2_1}), (\ref{3_1}) and (\ref{4_1})

\bibliographystyle{amsplain}

\begin{thebibliography}{10}

\bibitem{R} P.H. Rabinowitz, Periodic solutions of nonlinear hyperbolic partial differential equations, Comm. Pure Appl. Math., {\bf 22}, 145-205, (1967).

\bibitem{T} R. Temam, Infinite-Dimensional Dynamical Systems in Mechanics and Physics, Springer-Verlag, (1988).

\bibitem{T1}R. Temam, Behaviour at time $t=0$ of the solutions of semilinear evolutions equations, Jour. Diff. Eqs., {\bf 43}, 73-92, (1982).

\bibitem{TG} Ghidaglia and R. Temam, Attractors for damped nonlinear hyperbolic equations, J. Math. Pures Appl., {\bf 66}, 273-319, (1987).
\end{thebibliography}

\end{document}